\documentclass[11pt,a4paper]{amsart}
\usepackage{amsthm,amsmath}
\usepackage[all]{xy}
\title{On the Weak-Lefschetz property for Artinian Gorenstein algebras} 
\author{Alfio Ragusa
	\and  Giuseppe Zappal\`a}

\subjclass[2000]{13 D 40, 13 H 10}

\DeclareSymbolFont{rsfscript}{OMS}{rsfs}{m}{n}
\DeclareSymbolFontAlphabet{\mathrsfs}{rsfscript}

\DeclareSymbolFont{AMSb}{U}{msb}{m}{n}
\DeclareSymbolFontAlphabet{\mathbb}{AMSb}

\DeclareSymbolFont{eufrak}{U}{euf}{m}{n}
\DeclareSymbolFontAlphabet{\gothic}{eufrak}

\def\ac{\`}

\newcommand{\f}{\footnotesize}

\newcommand{\pp}{\mathbb P}

\newcommand\pf{\operatorname{pf}}

\newcommand\Gor{\operatorname{Gor}}

\newcommand\zz{{\mathbb Z}}

\newcommand{\BG}{\operatorname{BG}}

\newtheorem{thm}{Theorem}[section]
\newtheorem{lem}[thm]{Lemma}
\newtheorem{prp}[thm]{Proposition}
\newtheorem{cor}[thm]{Corollary}

\theoremstyle{definition}

\theoremstyle{remark}
\newtheorem{rem}[thm]{Remark}
\newtheorem{exm}[thm]{Example}

\newcounter{num}

\begin{document}


\begin{abstract}
We deal with the Weak Lefschetz property (WLP) for Artinian standard graded Gorenstein algebras of codimension $3.$  We prove that many Gorenstein sequences force the WLP for such algebras. Moreover for every Gorenstein sequence $H$ of codimension $3$ we found several Gorenstein Betti sequences compatible with $H$ which again force the WLP. Finally we show that for every Gorenstein Betti sequence the general Artinian standard graded Gorenstein algebra with such Betti sequence has the WLP.
\end{abstract}



\maketitle

\section*{Introduction}
\markboth{\it Introduction}{\it Introduction}

Let $R = k[x_1, . . . , x_r],$  where $k$ is a field of characteristic zero and let $I$ be an Artinian monomial complete intersection in $R,$ i.e. $I = (x^{a_1}, \ldots, x^{a_r}).$ Take a  general linear form $l\in R$. Then for any positive integers $d$ and $i,$ the  multiplication map defined by $l^d$, $\times l^d : (R/I)_i \to (R/I)_{i+d}$ has maximal rank. (In particular, this is true when $d = 1$).
The Weak and Strong Lefschetz properties were originated by the above result which was proved by different approaches by R. Stanley in \cite{St2} by algebraic topology,
by J.Watanabe in \cite{W} using representation theory, by Reid, Roberts and Roitman in \cite{RRR} by algebraic methods, by Herzog and Popescu (see \cite{HP}) by linear algebra and by Ikeda \cite{Ik}  using combinatorics. Since then the Weak and the Strong Lefschetz properties have been investigated very intensively. It is trivial to see that all Artinian standard graded algebras of codimension $2$ have the Weak Lefschetz Property, but already in codimension $3$ many questions are still open. A good survey on this subject could be the paper \cite{MN}. One of the most important results on this direction is the proof of the Weak Lefschetz Property for all Artinian complete intersections of codimension $3$, due to T.~Harima, J.C.~Migliore, U.~Nagel, J.~Watanabe, in \cite{HMNW}. Nevertheless, again in codimension $3,$ it is enough to consider almost complete intersection ideals,  i.e. perfect ideals of height $3$ and minimally generated by $4$ elements, to produce examples of Artinian standard graded algebras not enjoying the Weak Lefschetz Property (see \cite{BK}). But, as it is known, a natural class which generalizes Artinian complete intersections is the class of Artinian  Gorenstein standard graded algebras. Unfortunately, already in codimension $4$ there are  Artinian  Gorenstein standard graded algebras which do not have the Weak Lefschetz property (cfr. \cite{Ik}). Thus it remains to understand if every Artinian  Gorenstein standard graded algebras of codimension $3$ enjoys the Weak Lefschetz property. Results in special cases on this subject can be found in \cite{MZ} and \cite{A}. So the purpose of this paper is to study the Weak Lefschetz Property for the Artinian  Gorenstein standard graded algebras of codimension $3.$  Precisely, we produce Hilbert functions $H$ for which all Artinian  Gorenstein standard graded algebras with Hilbert function $H$ have the Weak Lefschetz Property, see Proposition \ref{hwlp}. The main results of the paper (see Theorem \ref{t1max}, Corollaries \ref{bmax} and \ref{bwlp}) say that for every Hilbert function $H$ there exists a Gorenstein Betti sequence $\beta_0$ compatible with $H,$ such that all Artinian Gorenstein standard graded algebras with Hilbert function is $H,$ whose Betti sequence is larger than or equal to $\beta_0$ have the Weak Lefschetz property. Because of this result we show that for every Gorenstein Betti sequence $\beta$ the generic Artinian algebra in the stratum of $\Gor H$ defined by $\beta$ has the Weak Lefschetz property.

\section{Notation and basic facts} 
\markboth{\it Notation and basic facts}
{\it Notation and basic facts}
Let $k$ be an infinite field and let $R:=k[x_1,x_2,x_3].$ We consider on $R$ the standard grading and we consider in it just homogeneous ideals, which we call simply ideals. 
\par
We will deal with Artinian Gorenstein algebras $A=R/I,$ precisely, Artinian standard graded Gorenstein algebras of codimension $3.$ 
\par
It is well known that $d_1\le\ldots\le d_{2m+1}$ are the degrees of a minimal set of generatora for a Gorenstein ideal of codimension $3$ iff the following Gaeta conditions hold (see \cite{Ga} for the general result and \cite{Di} for the Gorenstein version).
\begin{itemize}
	\item[1)]$\vartheta:=\sum_{i=1}^{2m+1}d_i/m$ is an integer;
	\item[2)]$\vartheta>d_i+d_{2m+3-i}$ for $2\le i\le m+1.$
\end{itemize}
So the resolution of $R/I$ is of the following type
 $$0\to R(-\vartheta)\to\bigoplus_{i=1}^{2m+1}R(d_i-\vartheta)\to\bigoplus_{i=1}^{2m+1}R(-d_i)\to R\to A\to 0.$$
Hence the Betti sequence of $A$ is 
 $$\beta(A)=(d_1,\ldots,d_{2m+1};\vartheta-d_{2m+1},\ldots,\vartheta-d_{1};\vartheta).$$
In the sequel when $A=R/I$ is an Artinian Gorenstein graded algebra, we will denote by $\nu_i(A)$ the number of minimal generators of $I$ belonging to $I_i$ and similarly by $\sigma_i(A)$ the number of minimal first syzygies of degree $i.$
With this notation the previous resolution can be write in the following way
\begin{equation}\label{r2}
 0\to R(-\vartheta)\to\bigoplus_{i}R(-i)^{\sigma_i(A)}\to\bigoplus_{i}R(-i)^{\nu_i(A)}\to R\to A\to 0.
\end{equation}
Note that $\beta(A)$ is uniquely determined by the knowledge of the degrees of minimal generators of $I,$ hence it is determined by the sequence $(\nu_i(A))_i.$ Therefore to give $\beta(A)$ is equivalent to assign the sequence $(\nu_i(A))_i.$
\par
In the sequel of the paper the Hilbert function and the graded Betti sequence of an Artinian standard graded Gorenstein algebra will be briefly called Gorenstein sequence and Gorenstein Betti sequence, respectively. 
\par
We will make frequently use of the following result (see Proposition 3.7 of \cite{RZ1} for the original version).
\begin{prp}\label{p35}
Let $A=R/I$ be an Artinian Gorenstein algebra with Hilbert function $H.$
Let us denote by $\vartheta$ the degree of the only second syzygy of $A,$ 
$s=\min\{i \in {\mathbb N} \ | \ I_i\ne 0 \}.$ Then
\begin{itemize}
\item [1)]$\vartheta=\max\{i \, \mid \, \Delta^{3}H_X(i)\ne 0\}.$
\item [2)]  $\nu_i(A) \ge -\Delta^{3}H(i)$ for $i <\vartheta$;
$\sigma_i(A)\ge \Delta^{3}H(i)$ for all $i.$
\item [3)] $\nu_s (A)=-\Delta^{3}H(s)$, $\nu_i(A) \le -\Delta^2 H(i)$
for $s+1 \le i \le \vartheta -s-1;$ and $\nu_i(A)=0$ for
$i\ge\vartheta-s$; consequently, $\sigma_{\vartheta-s}(A)=
\Delta^{3}H(\vartheta-s)$ and $\sigma_i (A)\le -\Delta^{2}H(i-1)$
for $s+1 \le i\le\vartheta-s-1.$
\end{itemize}
\end{prp}
Let $\varphi:\zz\to\zz$ be a function. We define $\varphi^+:\zz\to\zz$ to be the function such that $\varphi^+(i):=\max\{\varphi(i),0\},$ for every $i.$
\par
Let $H$ be the Hilbert function of an Artinian algebra $A$ of codimension $3$ and let $\vartheta-3$ be the socle degree of $A.$ 
Then $H$ is a Gorenstein sequence of codimension $3$ iff $H(i)=H(\vartheta-3-i),$ for every $i$ and $(\Delta H)^+$ is an $O$-sequence (see Theorem 4.2 in \cite{St1}).

\par

For every Gorenstein sequence $H$ we set 
$$u_H:=\min\{i\mid\Delta H(i)<0\}-\max\{i\mid\Delta H(i)>0\}-1;$$ 
$$s_H:=\min\{i\mid\Delta H(i)<i+1\};$$ $$t_H:=\max\{i\mid\Delta H(i)\ge0\};$$ $$\vartheta_H:=\max\{i\mid H(i)>0\}+3.$$
Since in the sequel the most interesting situations will happen when $u_H=0$ or $u_H=1$ we emphasize that in these cases $t_H+1=\frac{\vartheta_H-1}{2}$ or $t_H+1=\frac{\vartheta_H}{2}$ respectively.

According to Proposition \ref{p35}, we define two Betti sequences $\beta_{\min}(H)$ and $\beta_{\max}(H),$ compatible with $H.$
\par
If $u_H\ne 1$ or $u_H=1$ and $\Delta H(t_H-1)$ is even, then $\beta_{\min}(H)$ is defined by the sequence
$$\nu^{\min}_i(H):=\left\{\begin{array}{ll}
-\Delta^3H(i) & \text{for } i=s_H \\
\max\{0,-\Delta^3H(i)\} & \text{for } s_H+1\le i\le \vartheta_H-s_H-1 \\
0 & \text{otherwise }.
\end{array}\right.$$

If $u_H=1$ and $\Delta H(t_H-1)$ is odd, then $\beta_{\min}(H)$ is defined by the sequence
$$\nu^{\min}_i(H):=\left\{\begin{array}{ll}
-\Delta^3H(i) & \text{for } i=s_H \\
\max\{0,-\Delta^3H(i)\} & \text{for } s_H+1\le i\le \vartheta_H-s_H-1,\,i\ne\frac{\vartheta_H}{2} \\
1 & \text{for }i=\frac{\vartheta_H}{2}\\
0 & \text{otherwise }.
\end{array}\right.$$

$\beta_{\max}(H)$ is defined by the sequence
$$\nu^{\max}_i(H):=\left\{\begin{array}{ll}
-\Delta^2H(i)+1 & \text{for } i=s_H \\
\max\{0,-\Delta^2H(i)\} & \text{otherwise}.
\end{array}\right.$$

Let $\beta_1$ and $\beta_2$ be two Gorenstein Betti sequences, associated respectively to the sequences $(\nu'_i)_i$ and $(\nu''_i)_i.$
We say that $\beta_1\le\beta_2$ iff $\nu'_i\le\nu''_i,$ for every $i.$ 
\par
Note that if $\beta_1\le\beta_2$ according to the previous definition, then using the same notation as in the resolution (\ref{r2}) it will be $\sigma'_i\le\sigma''_i$ for every $i$ and 
$\vartheta'\le\vartheta''.$
\begin{prp}\label{bg}
Let $H$ be a Gorenstein sequence of codimension $3$ and let $\BG_H$ be the poset of all Gorenstein Betti sequences compatible with $H.$
\begin{itemize}
	\item[1)] $\BG_H$ has only one minimal and only one maximal element, respectively $\beta_{\min}(H)$ and $\beta_{\max}(H).$
	\item[2)] If $d_1\le\ldots\le d_{2m+1}$ are the degrees of the minimal generators associated to $\beta_{\max}(H),$ then $m=d_1$ and $d_i+d_{2m+3-i}+1=\vartheta_H,$ for every $2\le i\le m+1.$
	\item[3)] Each element in $\BG_H$ can be obtained by $\beta_{\max}(H)$ by performing couples of cancellations of the type $i,\vartheta_H-i.$ All these cancellations are allowed.
\end{itemize}
\end{prp}
\begin{proof}
\begin{itemize}
	\item[1)] See~\cite{RZ1} Proposition 3.7 and Remark 3.8.
	\item[2)] By definition of $\beta_{\max}(H),$ $2m+1=\sum_i\max\{0,-\Delta^2H(i)\}+1.$ Now
	$$\Delta^2H(i)=\left\{\begin{array}{ll}
 1 & \text{for } 0\le i\le d_1-1 \text{ and }\vartheta_H-d_1\le i\le \vartheta_H-1 \\
 \le 0 & \text{for } d_1\le i\le\vartheta_H-d_1-1.
 \end{array}\right.$$
But $2m=-\sum_{i=d_1}^{\vartheta_H-d_1-1}\Delta^2H(i)=2d_1$ recalling that $\sum_{i}\Delta^2H(i)=0.$
\par
Moreover 
 $$\sum_{i=1}^{2m+1}d_i-\sum_{i=1}^{2m+1}(\vartheta_H-d_i)+\vartheta_H=0$$
from which we get
 $$\sum_{i=2}^{2m+1}[(\vartheta_H-d_i)-d_{2m+3-i}]=2d_1;$$
but the first side consists of the sum of $2d_1$ positive integers, therefore each number must be equal to $1.$
	\item[3)] This follows from Remark 3.8 in \cite{RZ1}.
\end{itemize}
\end{proof}

\section{WLP for Gorenstein algebras} 
\markboth{\it WLP for Gorenstein algebras}
{\it WLP for Gorenstein algebras}
The aim of this paper is to study the Weak-Lefschetz property (WLP) for Artinian standard graded Gorenstein algebras of codimension $3.$
We recall that an Artinian standard graded $S$-algebra $A$ has the WLP if there exists a linear form $l\in S$ such that for every $i>0$ the multiplication map (of $k$-vector spaces) $l:A_{i-1}\to A_i$ has maximal rank.
It is well known that in codimension $\ge 4,$ not all Artinian standard graded Gorenstein algebras have the WLP (see for instance \cite{Ik}) but in codimension $3$ the problem is still wide open.
We will show that given a Gorenstein sequence $H$ (or a Gorenstein Betti sequence $\beta$) of codimension $3$ the WLP holds for the generic Artinian standard graded Gorenstein algebra with Hilbert function $H$ (respectively with Betti sequence $\beta$). Moreover, given again $H$ as before, we will see that, in some sense, the most \lq\lq special\rq\rq\ and the \lq\lq general\rq\rq\ Artinian standard graded Gorenstein algebras have the WLP.

We start with this observation.

\begin{prp}\label{postot}
Let $A=R/I$ be an Artinian standard graded Gorenstein algebra, with socle degree $\vartheta-3$ and Hilbert function $H.$ Then $A$ enjoys the WLP iff there exists a linear form $l\in R$ such that the multiplication map  
$l:A_{t-1}\to A_t$ is injective, where $t=t_H.$
\end{prp}
\begin{proof}
By hypothesis and Proposition 2.1, b) in \cite{MMN}, we get that the multiplication map $l:A_{i-1}\to A_i$ is injective for $i\le t.$ Now since $t\ge\frac{\vartheta-3}{2},$ by Gorenstein duality we deduce that the multiplication map $l:A_{i-1}\to A_i$ is surjective for $i>t,$ so it has maximal rank in any degree.
\end{proof}

Next result will be very useful in the rest of the paper.

\begin{lem}\label{fj}
Let $A=R/I$ be an Artinian standard graded algebra.
Let us suppose that there exists an integer $t,$ a perfect ideal $J\subset R$ of height $\le 2$ and a form $f\in R$ (eventually a constant) such that $I_i=(fJ)_i$ for $i=t-1,t.$
Then there exists a linear form $l\in R$ such that the multiplication map $l:A_{t-1}\to A_t$ is injective.
\end{lem}
\begin{proof}
The ring $R/J$ is Cohen-Macaulay of Krull dimension $\ge 1.$ It is enough to choose a linear form in $R$ such that $l$ is not a factor of $f$ and it is regular in $R/J.$
\end{proof}

We collect in the next proposition some cases in which the Hilbert function $H$ forces the WLP for every Artinian standard graded Gorenstein algebras having $H$ as Hilbert function.

\begin{prp}\label{hwlp}
Let $H$ be a Gorenstein sequence of codimension $3$ and let $u=u_H$ and $t=t_H$ as in the previous section. 
If $H$ satisfies one of the following conditions
\begin{itemize}
	\item[1)] $\Delta H(t-1)=\Delta H(t);$
	\item[2)] $\Delta H(t-1)=1$ and $\Delta H(t)=0;$
	\item[3)] $u=0$ and $\Delta H(t)\ge t-1;$	
\end{itemize}
then every Artinian standard graded Gorenstein algebra with Hilbert function $H$ has the WLP.
\end{prp}
\begin{proof}
\begin{itemize}
\item[1)] Note that for Theorem 3.1 and Corollary 3.4 in \cite{RZ1}, every Artinian standard graded Gorenstein algebra $A=R/I$ with Hilbert function $H$ is such that 
$I_{\le t}=fJ,$ where $f$ is a form of degree $\Delta H(t)$ and $J$ is perfect ideal of height $\le 2.$ So, applying Lemma \ref{fj}, we are done.
\item[2)] In this case $\sum_{i=0}^t\Delta^3H(i)=\Delta^2H(t)=-1.$ Using the same arguments of proof of Theorem 3.1 in \cite{RZ1}, if we set $p$ the number of the minimal generators of $I$ of degree $\le t$ and $q$ the number of the minimal first syzygies of $I$ of degree $\le t+1,$ we get that $p-q=1.$ Taking a look at the Buchsbaum-Eisenbud matrix (see \cite{BE}) of $A,$ again as in the proof of Theorem 3.1 in \cite{RZ1}, we get that $I_{\le t}$ is a perfect ideal of height $2,$ so by Lemma \ref{fj} and Proposition \ref{postot}, $A$ enjoys the WLP.
\item[3)] Since $(\Delta H)^+$ is an $O$-sequence, in this case the ideal $I_{\le t}$ is minimally generated by at most two elements. Therefore $I_{\le t}=0$ or $I_{\le t}$ is principal or 
$I_{\le t}=fJ$ where $J$ is generated by a regular sequence of length $2.$ By Lemma \ref{fj} the conclusion follows.
\end{itemize}
\end{proof}
\begin{rem}\label{u2}
Note that if $u_H\ge 2,$ then the hypothesis $1$ of Proposition \ref{hwlp} is satisfied. Therefore, since when $u_H\ge 2$ the WLP holds, in the sequel we will treat only the cases $u_H=0,1.$
\end{rem}
Now we concentrate on the relationship between the WLP and the Betti sequences of the Artinian standard graded Gorenstein algebras.
\par
One key result is the following theorem.

\begin{thm}\label{t1max}
Let $A=R/I$ be an Artinian standard graded Gorenstein algebra. Let $H$ be the Hilbert function of $A$ and let $t=t_H.$ If $\nu_{t+1}(A)=\nu^{\max}_{t+1}(H)$ then $A$ enjoys the WLP.
\end{thm}
\begin{proof}
By Remark \ref{u2} we can suppose that $u:=u_H\le 1.$
Let $s:=s_H$ and $\vartheta:=\vartheta_H.$ Note that $s-1\le t\le\vartheta-s-2$ and when $t=s-1$ we have that $A_{i}=R_{i}$ for $i\le t,$ so the WLP is trivial by Proposition \ref{postot}.
So we can assume $s+1\le t+1\le\vartheta-s-1.$ If $\nu_{t+1}(A)>\nu^{\min}_{t+1}(H),$ applying Theorem 4.4 in \cite{RZ2}, we get that $I_{\le t}=fJ$ where $f$ is a form and $J$ is perfect of height $\le 2.$ So the hypotheses of Lemma \ref{fj} are satisfied and by Proposition \ref{postot} $A$ has the WLP. Let us suppose now that $\nu_{t+1}(A)=\nu^{\min}_{t+1}(H).$ 
If $u=0$ then
 $$\nu^{\min}_{t+1}(H)=\max\{0,-\Delta^3H(t+1)\}=\max\{0,-\Delta^2H(t+1)\}=-\Delta^2H(t+1)>0$$
i.e. $\Delta^3H(t+1)=\Delta^2H(t+1);$ this implies that $\Delta H(t-1)=\Delta H(t).$ So by Proposition \ref{hwlp} 1) we get that $A$ has the WLP.
\par
If $u=1$ the assumption $\nu_{t+1}(A)=\nu^{\min}_{t+1}(H)$ can happen only if $\Delta H(t-1)=1,$ so the conclusion follows by Proposition \ref{hwlp} item $2.$
\end{proof}

\begin{cor}\label{bmax}
Let $A=R/I$ be an Artinian standard graded Gorenstein algebra. Let $\beta$ be the Betti sequence of $A$ and $H$ be the Hilbert function of $A.$ If $\beta=\beta_{\max}(H)$ then $A$ enjoys the WLP. In particular for every Gorenstein sequence $H$ of codimension $3,$ there is an Artinian standard graded Gorenstein algebra whose Hilbert function is $H$ and having the WLP.
\end{cor}
\begin{proof}
This is just a particular case of Theorem \ref{t1max}.
\end{proof}

\begin{cor}\label{bwlp}
Let $H$ be a Gorenstein sequence of codimension $3.$ Then there exists a Gorenstein Betti sequence $\beta_0$ compatible with $H$ such that for every Gorenstein Betti sequence $\beta\in\BG_H,$ $\beta\ge\beta_0$ all Artinian standard graded Gorenstein algebras having Betti sequence $\beta$ enjoy the WLP.
\end{cor}
\begin{proof}
Let us consider the sequence $\nu_0$ defined as follows
$$(\nu_0)_i:=\left\{\begin{array}{ll}
\nu^{\max}_i(H) & \text{for } i=t_H+1,\vartheta_H-t_H-1 \\
\nu^{\min}_i(H) & \text{otherwise }
\end{array}.\right.$$
Let $\beta_0$ be the Betti sequence corresponding to $\nu_0.$ Note that $\beta_0$ is admissible for an Artinian standard graded Gorenstein algebra by Proposition \ref{bg}, item $3.$ So if $A$ is an 
Artinian standard graded Gorenstein algebra such that $\beta(A)\ge\beta_0$ then it has the WLP by Theorem \ref{t1max}.
\end{proof}

\begin{exm}
Let us consider the following Gorenstein sequence
\begin{multline*}
 H=(1,3,6,10,15,21,28,36,45,55,66,76,84,89,90, \\ 89,84,76,66,55,45,36,28,21,15,10,6,3,1).
\end{multline*}
Then $\beta^{\max}(H)$ is defined by the sequence
\begin{multline*}
 \nu_{11}=2,\,\nu_{12}=2,\,\nu_{13}=3,\,\nu_{14}=4,\,\nu_{15}=2,\,\nu_{16}=4,\,\\ \nu_{17}=3,\,\nu_{18}=2,\,\nu_{19}=1,
\end{multline*}
and $\beta_0$ is defined by the sequence
\begin{equation*}
 \nu_{11}=2,\,\nu_{12}=1,\,\nu_{13}=1,\,\nu_{14}=1,\,\nu_{15}=2,\,\nu_{16}=4.
\end{equation*}
\end{exm}

\begin{rem}
Note that if $k$ is algebraically closed by the irreducibility of $\Gor H,$ the previous results imply that the generic element in $\Gor H$ has the WLP.
\end{rem}

\begin{thm}\label{canc}
Let $A=R/I$ be an Artinian standard graded Gorenstein algebra, enjoying the WLP, whose Hilbert function is $H$ and Betti sequence is $\beta.$ Let $\beta'\in\BG_H,$ $\beta'<\beta$ such that $]\beta',\beta[=\emptyset.$ Then there exists a $1$-dimensional family of Artinian standard graded Gorenstein algebras with Betti sequence $\beta'$ and the WLP.
\end{thm}
\begin{proof}
Since $]\beta',\beta[=\emptyset$ we get $\beta'$ from $\beta$ by performing a couple of cancellations of degrees $r$ and $\vartheta_H-r$ among the degrees of minimal generators and the degrees of first syzygies (see Remark 3.8 in \cite{RZ1}). Let $M$ be an alternating matrix whose submaximal pfaffians minimally generate $I.$ Let $M_{\widehat i}$ be the submatrix obtained by $M=(m_{ij})$ by deleting the $i$-th row and the $i$-th column. Let $p_i=\pf M_{\widehat i}.$ Then $I$ is minimally generated by $p_1,\ldots,p_{2m+1}.$ So we have a minimal generator $p_a$ of degree $r$ and a minimal generator $p_b$ of degree $\vartheta_H-r$ and two minimal first syzygies of the same degrees. So, by minimality, $m_{ab}=m_{ba}=0.$ Let $\lambda\in k$ and let $M_{(\lambda)}=(m'_{ij})$ be the alternating matrix such that $m'_{ab}=\lambda,$ $m'_{ba}=-\lambda$ and $m'_{ij}=m_{ij}$ otherwise. Let $I_{(\lambda)}$ be the ideal generated by the submaximal pfaffians of $M_{(\lambda)}.$ So $A_{(\lambda)}=R/I_{(\lambda)}$ is an Artinian standard graded Gorenstein algebra with Betti sequence $\beta$ for $\lambda=0$ and Betti sequence $\beta'$ for $\lambda\ne 0.$ Since $A$ has the WLP there exists $l\in R_1$ such that the multiplication map $l:A_{i-1}\to A_{i}$ has maximal rank for every $i.$ Now let us consider the multiplication map $l:(A_{(\lambda)})_{t-1}\to (A_{(\lambda)})_{t},$ where $t=t_H.$ Of course the rank of the latter map depends on $\lambda$ and it is maximal for $\lambda=0,$ since $A_{(0)}=A.$ So this rank is maximal for the generic $\lambda\in k,$ therefore the generic $A_{(\lambda)}$ has the WLP.
\end{proof}

\begin{cor}\label{bgen}
Let $\beta$ be any Gorenstein Betti sequence of codimension $3.$ Then there are infinitely many Artinian standard graded Gorenstein algebra with Betti sequence $\beta$ and enjoying the WLP. In particular, if $k$ is algebraically closed, the generic element of $\Gor\beta$ has the WLP.
\end{cor}
\begin{proof}
Let $H$ be the Hilbert function corresponding to $\beta.$ By Corollary \ref{bmax} every Artinian standard graded Gorenstein algebra with Betti sequence $\beta_{\max}(H)$ has the WLP. Now applying iteratively Theorem \ref{canc} we can reach $\beta$ from $\beta_{\max}(H),$ obtaining families of Artinian standard graded Gorenstein algebra with Betti sequence $\beta$ and enjoying the WLP. The last assertion follows from the irreducibility of the stratum of $\Gor H$ defined by $\beta$ (see \cite{Di} section $3$).
\end{proof}

\vspace{1cm}
{\f
{\sc (A. Ragusa) Dip. di Matematica e Informatica, Universit\`a di Catania,\\
                  Viale A. Doria 6, 95125 Catania, Italy}\par
{\it E-mail address: }{\tt ragusa@dmi.unict.it} \par
{\it Fax number: }{\f +39095330094} \par
\vspace{.3cm}
{\sc (G. Zappal\`a) Dip. di Matematica e Informatica, Universit\`a di Catania,\\
                  Viale A. Doria 6, 95125 Catania, Italy}\par
{\it E-mail address: }{\tt zappalag@dmi.unict.it} \par
{\it Fax number: }{\f +39095330094}
}

\end{document}